\documentclass[a4paper,10pt]{amsart}
\usepackage{amssymb}
\usepackage[utf8]{inputenc}
\usepackage[T1]{fontenc}                
\usepackage[french,english]{babel}
\usepackage[hidelinks]{hyperref}
\usepackage{mathrsfs}

\usepackage{color}
\usepackage[pdftex]{graphicx}

\theoremstyle{plain}
\newtheorem{Theorem}{Theorem}
\newtheorem{Lemma}{Lemma}

\newtheorem*{Theorem*}{Theorem}

\theoremstyle{definition}

\theoremstyle{remark}
\newtheorem{Remark}[Lemma]{Remark}

\newcommand{\R}{\mathbb{R}}     
\newcommand{\HH}{\mathbb{H}}

\newcommand{\W}{\bb W}
\newcommand{\V}{\bb V}
\newcommand{\G}{\bb G}

\newcommand{\bb}[1]{\mathbb{#1}}

\newcommand{\hel}
{
\hskip2.5pt{\vrule height6pt width.5pt depth0pt}
\hskip-.2pt\vbox{\hrule height.5pt width6pt depth0pt}
\,
}

\title{Nowhere differentiable intrinsic Lipschitz graphs}

\author[Julia]{Antoine Julia}
	\address[A.~Julia]{D\'epartment de Math\'ematiques d'Orsay, Universit\'e Paris-Saclay, 91405, Orsay, France}
	\email{antoine.julia@u-psud.fr}

\author[Nicolussi~Golo]{Sebastiano Nicolussi Golo}
	\address[Nicolussi Golo]{Department of Mathematics and Statistics, 40014 University of Jyväskylä, Finland}	\email{sebastiano2.72@gmail.com}
        
\author[Vittone]{Davide Vittone}
	\address[D.~Vittone]{Dipartimento di Matematica ``T.Levi-Civita'', via Trieste 63, 35121 Padova, Italy.}
	\email{vittone@math.unipd.it}

\thanks{A.J.~has been supported by the Simons Foundation Wave Project.  S.N.G.~has been supported by the Academy of Finland (grant 322898 ``Sub-Riemannian Geometry via  Metric-geometry and Lie-group Theory''). D.V.~has been supported by  FFABR 2017 of MIUR (Italy) and by GNAMPA of INdAM (Italy). All three authors have been supported by the University of Padova STARS Project ``Sub-Riemannian Geometry and Geometric Measure Theory Issues: Old and New''.}


\date{\today}

\begin{document}
\begin{abstract}
 We construct  intrinsic Lipschitz graphs in Carnot groups with the property that, at every point, there exist infinitely many different blow-up limits, none of which is a homogeneous subgroup. This provides counterexamples to a Rademacher theorem for intrinsic Lipschitz graphs.
\end{abstract}
\maketitle

The notion of Lipschitz submanifolds in sub-Riemannian geometry was introduced, at least in the setting of Carnot groups, by B.~Franchi, R.~Serapioni and F.~Serra Cassano in a series of seminal papers~\cite{FSSCJNCA,MR2836591,FranchiSerapioni2016Intrinsic} through the theory of {\em intrinsic Lipschitz graphs}. One of the main open questions  concerns the differentiability properties for such graphs: in this paper we provide  examples of intrinsic Lipschitz graphs of codimension 2 (or higher) that are nowhere differentiable, i.e., that admit no homogeneous tangent subgroup at any point.

\medskip

Recall that a Carnot group $\G$ is a connected, simply connected and nilpotent Lie group whose Lie algebra is stratified, i.e., it can be decomposed as the direct sum $\oplus_{j=1}^s V_j$ of subspaces such that
\[
V_{j+1}=[V_1,V_j]\text{ for every }j=1,\dots,s-1,\qquad [V_1,V_s]=\{0\},\qquad V_s\neq\{0\}.
\]
We shall identify the group $\G$ with its Lie algebra via the exponential map $\exp:\oplus_{j=1}^s V_j\to\G$, which is a diffeomorphism. In this way, for $\lambda>0$ one can introduce the homogeneous dilations $\delta_\lambda:\G\to\G$ as the group automorphisms defined by $\delta_\lambda(p)=\lambda^j p$ for every $p\in V_j$. A subgroup of $\G$ is said to be homogeneous if it is dilation-invariant. Assume that a splitting $\G=\W\V$ of $\G$  as the product of homogeneous and complementary   (i.e., such that $\W\cap\V=\{0\}$) subgroups is fixed; we say that a function $\phi:\bb W\to\bb V$ \emph{intrinsic Lipschitz} if there is an open nonempty cone $U$ such that $\V\setminus\{0\}\subset U$ and
\[
pU\cap\Gamma_\phi=\emptyset\qquad\text{for all $p\in\Gamma_\phi$,}
\]
where  $\Gamma_\phi=\{w\phi(w):w\in\bb W\}$ is the intrinsic graph of $\phi$.
We say that a set $\Sigma\subset\G$ is a {\em blow-up} of $\Gamma_\phi$ at $\hat p=\hat w\phi(\hat w)$ if there exists a sequence $(\lambda_n)_n$ such that $\lambda_n\to+\infty$ and the limit
\[
\lim_{n\to\infty} \delta_{\lambda_n}( \hat p^{-1}\Gamma_\phi)=\Sigma
\]
holds with respect to the local Hausdorff convergence. It is worth recalling that, if $\phi$ is intrinsic Lipschitz, then every blow-up  is automatically the intrinsic Lipschitz graph of a map $\W\to\V$. Eventually, we say that $\phi$ is \emph{intrinsically differentiable} at $\hat w\in\bb W$ if the blow-up of $\Gamma_\phi$ at $\hat p=\hat w\phi(\hat w)$ is unique and it is a homogeneous subgroup of $\bb G$. See~\cite{2020arXiv200402520J} for details.

\medskip

We say that a group $\bb G$ along with a splitting $\bb W \bb V$ satisfies an \emph{intrinsic Ra\-de\-ma\-cher Theorem}  if all intrinsic Lipschitz maps $\phi : \bb W\to \bb V$ are intrinsically differentiable almost everywhere (that is, for almost all points of $\bb W$ equipped with its Haar measure). It was proved in \cite{MR2836591} that this is the case when $\bb V \simeq \R$ and $\bb G$ is  of step two; other partial results for graphs with codimension 1 ($\V\simeq\R$) are contained in~\cite{FrMaSe2014Differentiability} and~\cite{LDM_semigenerated}.
If $\bb V$ is a normal subgroup, the Rademacher Theorem has been proved for general $\bb G$ by G. Antonelli and A. Merlo in \cite{2020arXiv200602782A}. Recently,  the third named author~\cite{2020arXiv200714286V} proved that Heisenberg groups (with any splitting) satisfy an intrinsic Rademacher Theorem.
The question has been open for a long time if $\bb G$ is the Engel group (which has step 3) and $\bb V \simeq \R$ (see \cite{MR2496564}). 
In this paper we prove a result in the negative direction: namely,  we provide  examples of  intrinsic Lipschitz graphs that are nowhere intrinsically differentiable. Let us state our main result:

\begin{Theorem}\label{thm:general}
	Let $\G$ be a Carnot group with stratification $\bigoplus_{j=1}^s V_j$.
	Let $\W\V$ be a splitting of $\G$ such that $\W\cap V_2\not\subset[\W,\W]$ and there exists $v_0\in\V\cap V_1$  such that $v_0\neq 0$ and $[v_0,\W]=0$.
	Then there is an intrinsic Lipschitz function $\phi:\W\to\V$ that is nowhere intrinsically differentiable.

Moreover, $\phi$ can be constructed in such a way that, for every $p\in\Gamma_\phi$, the following properties hold:
\begin{itemize}
\item[(a)] there exist infinitely many different blow-ups of $\Gamma_\phi$ at $p$,
\item[(b)] no blow-up of $\Gamma_\phi$ at $p$ is a homogeneous subgroup.
\end{itemize}
      \end{Theorem}
    
    The proof of Theorem~\ref{thm:general} is postponed in order to first provide some comments.
      
\begin{Remark}
The simplest example of a Carnot group where Theorem~\ref{thm:general} applies is $\G=\HH\times\R$, where $\HH$ is the first Heisenberg group. As customary, we consider generators $X,Y,T$ of the Lie algebra of $\HH$ such that $[X,Y]=T,[X,T]=[Y,T]=0$ and fix the exponential coordinates $(x,y,t)=\exp(xX+yY+tT)$. Using coordinates $(x,y,t,r)$ on $\HH\times\R$ with $r\in\R$, we can consider the splitting $\HH\times\R=\W\V$ given by the vertical subgroup $\W=\{x=r=0\}$ of $\HH$ and the horizontal abelian subgroup $\V=\{y=t=0\}$.
	Then $V_2\cap\W\not\subset [\W,\W]=\{0\} $ and $v_0=(0,0,0,1)$ commutes with $\W$.
	Hence, this splitting of $\HH\times\R$ satisfies the conditions of Theorem~\ref{thm:general} and it does not satisfy an intrinsic Rademacher Theorem.
        
\medskip
	
It is worth observing that, in this setting, the map $\phi:\W\to\V$ provided in the proof of Theorem~\ref{thm:general} takes the form $\phi(y,t)=(0,u(t))$, where $u$ is the $\frac 12$-H\"older continuous function constructed in the Appendix.	In particular, the intrinsic graph $\Gamma_\phi$ is the set $\{(0,y,t,u(t)):y,t\in\R\}$ and it is contained in the Abelian subgroup $\W\times\R$. One of the properties of $u$ is that the limit
\[
\lim_{s\to t}\dfrac{\vert u(t)-u(s)\vert }{\sqrt{\vert t-s\vert}}
\]
does not exists at any $t\in\R$ and this is the ultimate reason for the non-dif\-fe\-ren\-tia\-bi\-li\-ty of $\phi$.      

 \medskip

Similar counterexamples can be constructed in any codimension $k\geq 2$: in fact one can consider $\HH^{k-1}\times \R=(\R^{k-1}_x\times\R^{k-1}_y\times\R_t)\times\R_r$ with splitting $\W\V$ defined by $\W=\{x=0,r=0\},\V=\{y=0,t=0\}$. It can be easily checked that the map $\phi(y,t)=(0,u(t))$ defines an intrinsic Lipschitz graph of codimension $k$ for which the properties (a) and (b) in Theorem~\ref{thm:general} hold at every point.
\end{Remark}
      
 \begin{Remark}

The measure $\mu=\mathcal H^{d}\hel \Gamma_\phi$, where $d$ is the Hausdorff dimension of $\W$ and $\mathcal H^d$ is the $d$-dimensional Hausdorff measure, does not have a unique tangent measure at any point.
  Indeed, firstly, any tangent measure of $\mu$ is supported on a blow-up of $\Gamma_\phi$.
 Secondly, by~\cite[Theorem 3.9]{FranchiSerapioni2016Intrinsic}, $\mu$ and all its dilations are uniformly $d$-Ahlfors regular, and thus any tangent measure of $\mu$ is  $d$-Ahlfors regular. 
 We then conclude that if $\mu_1$ and $\mu_2$ are two tangent measures of $\mu$ supported on different blow-ups of $\Gamma_\phi$, then they are two distinct measures. 
 Since blow-ups of $\Gamma_\phi$ are not unique, so are tangent measures. Observe also that no tangent measure can be flat, i.e., supported on a homogeneous subgroup.  
In particular, $\Gamma_\phi$ is purely $C^1_H$-unrectifiable, i.e., $\mathcal H^d(\Gamma_\phi\cap\Sigma)=0$ for every submanifold $\Sigma$ of class $C^1_H$ (see e.g.~\cite[\S~2.5 and~6.1]{antonelli2020rectifiable}). 
\end{Remark}

\begin{Remark}
	If $\W$ is a homogeneous subgroup of $\G$ with codimension 1, then the conditions of Theorem~\ref{thm:general} cannot be met because $\bigoplus_{j=2}^sV_j = [\W,\W] + [\W,\V]$. Actually, intrinsic Lipschitz graphs of codimension 1 are boundaries of sets with finite perimeter in $\G$ (see e.g.~\cite[Theorem~1.2]{VSNS}), hence at almost every point they possess at least one blow-up which is a homogeneous subgroup of codimension 1, see~\cite{MR2496564}. Therefore, any possible counterexample to the Rademacher Theorem in codimension 1 cannot be as striking as the one provided by Theorem~\ref{thm:general}, in the sense that property (b) cannot hold on a  set with positive measure.
      \end{Remark}

\begin{Remark}
	Following the same proof strategy, one can extend Theorem~\ref{thm:general} to the case $\W\cap V_j\not\subset[\W,\W]$ for some $j>2$ and $v_0\in V_k\cap\V\setminus\{0\}$ with $k<j$ and $[v_0,\W]=0$, by taking a $k/j$-Hölder analogue of the function $u$ constructed in the appendix.
\end{Remark}

\begin{proof}[Proof of Theorem~\ref{thm:general}]
	Let $\beta:\W\to\R$ be a nonzero linear function such that  $\W\cap V_j\subset\ker\beta$ whenever $j\neq2$ and $[\W,\W]\subset\ker\beta$; such a $\beta$ exists\footnote{For instance, one can consider $\beta(x)=\langle x,w_0 \rangle$ for some $w_0\in(\W\cap V_2)\setminus[\W,\W]$ and a scalar product on $\W$ adapted to the grading $\bigoplus_{j=1}^s \W\cap V_j$ of $\W$.} because $\W\cap V_2\not\subset[\W,\W]$.
	Notice that such a function $\beta$ is in fact a group morphism $\W\to\R$.

        \medskip

        Consider a 1/2-H\"older continuous function $u:\R\to\R$ with the following properties. First, the difference quotients 
\[
\Delta(s,t)=\frac{u(s)-u(t)}{\text{sgn}(s-t)\vert s-t\vert^{1/2}}
\]
are bounded, namely,
\begin{equation}\label{eq_12Holder1}
|\Delta(s,t)|\leq 1\qquad\text{for every $s,t\in\R$.}
\end{equation}
	Second, there exist $c_1>0$ and $c_2>0$ such that, for every $t_0\in\R$ and $\delta\in(0,1\,]$, there exist $s_1,s_2\in\R$ such that
	\begin{equation}\label{eq_12Holder2}
	\begin{array}{l}
	\text{sgn}(s_1-t_0)=\text{sgn}(s_2-t_0)\\
	c_1\delta\leq|s_1-t_0|\leq \delta\\
	c_1\delta\leq|s_2-t_0|\leq \delta\\
	|\Delta(s_1,t_0)-\Delta(s_2,t_0)|\geq c_2.
	\end{array}
	\end{equation}
	Such a function exists, as we show in the appendix.

        \medskip
	
	We can then define $\phi:\W\to\V$ as
	\[
	\phi(w) = u(\beta(w)) v_0 .
	\]
	Notice that the condition $[v_0,\W]=0$ implies that 
		\begin{equation}\label{eq_Abelian}
	vw=wv\qquad\text{for all $w\in\W$ and $v\in\R v_0$.}
		\end{equation}
	Therefore, the intrinsic graph of $\phi$ is the set of points $w\phi(w) = w + u(\beta(w)) v_0$ for $w\in\W$.

        \medskip
        
	{\bf Claim 1:} The map $\phi$ is intrinsic Lipschitz.\\
	Fix a homogeneous norm $\|\cdot\|$ on $\G$.
	Notice that, since $\beta(\delta_\lambda x) = \lambda^2\beta(x)$ for all $x\in\W$,  there is a constant $C$ such that $\vert \beta(x) \vert \le C \|x\|^2 $, for all $x\in\W$.
	We check that $\Gamma_\phi$ has the cone property for the cone
	\[
	  U = \{ wv: w\in\W,\ v\in\V, \|v\| > 2\sqrt{C}\|v_0\| \|w\|  \} .
	\]
	Given $\hat w, w\in\W$, by~\eqref{eq_Abelian} we have
	$(\hat w\phi(\hat w))^{-1}(w\phi(w)) = (\hat w^{-1}w)(\phi(\hat w)^{-1}\phi(w))$ and
	\begin{align*}
	\| \phi(\hat w)^{-1}\phi(w) \| 
	&= \vert u(\beta(w)) - u(\beta(\hat w)) \vert \|v_0\| 
	\le \vert \beta(w)-\beta(\hat w) \vert^{1/2} \|v_0\|  \\
	&= \vert \beta(\hat w^{-1}w) \vert^{1/2} \|v_0\| 
	\le \sqrt{C} \|\hat w^{-1}w\| \|v_0\|  .
	\end{align*}
	Thus, $ (\hat w\phi(\hat w))^{-1}\Gamma_\phi \cap U = \emptyset$ for all $\hat w\in\W$, i.e., $\Gamma_\phi$ is an intrinsic Lipschitz graph.
	
\medskip
	
	{\bf Claim 2:} for  $p\in\Gamma_\phi$, none of the blow-ups of $\Gamma_\phi$ at $p$ is a homogeneous subgroup.\\
We first observe that, if $\V_0\subset \V\cap V_1$ is the horizontal subgroup generated by $v_0$ and $L:\W\to\V_0$	parametrizes a homogeneous subgroup $\Gamma_L$ of $\G$, then $L|_{\W\cap V_2}=0$. Indeed, the homogeneity of $\Gamma_L$ implies that for every $h>0$ and $w\in\W\cap V_2$ one has $L(2w)=\sqrt 2\,L(w)$, while the fact that $\Gamma_L$ is a subgroup  (plus the fact that $\V_0$ and $\W$ commute) gives $L(2 w)=2L(w)$. This proves that $L=0$ on $\W\cap V_2$.
	
\medskip

		We now prove the claim. Assume by contradiction that there exist $\hat p=\hat w\phi(\hat w)\in\Gamma_\phi$, a map $L:\W\to\V$ such that the intrinsic graph $\Gamma_L$ of $L$ is a homogeneous subgroup and a sequence $(\lambda_n)_n$ with $\lambda_n\to+\infty$,  and 
		\[
\lim_{n\to\infty} \delta_{\lambda_n}( \hat p^{-1}\Gamma_\phi)=\Gamma_L.
\]
Observe that for every $w\in\W$ and every $n$
	\begin{align*}
	\delta_{\lambda_n}((\hat w\phi(\hat w))^{-1} ( w\phi( w) ) )
	&= \delta_{\lambda_n}(\hat w^{-1} w \phi(\hat w)^{-1}\phi( w) ) \\
	&= \delta_{\lambda_n}(\hat w^{-1} w) \left( \frac{ u(\beta(w)) - u(\beta(\hat w)) }{ 1/{\lambda_n}} v_0 \right).
	\end{align*}
	If we set $w=\hat w\delta_{1/{\lambda_n}}w'$, then $\beta(w) = \beta(\hat w) + \beta(w')/\lambda_n^2$.
	Therefore, the set $\delta_{\lambda_n}(\hat p^{-1}\Gamma_\phi)$ is the intrinsic graph of the function from $\W$ to $\V$ given by
	\[
	\phi_{\hat p,\lambda_n}(w') = \frac{ u(\beta(\hat w)+\beta(w')/\lambda_n^2) - u(\beta(\hat w)) }{ 1/\lambda_n } v_0 .
	\]
	Since the maps $\phi_{\hat p,\lambda_n}$ take values in $\V_0$, $L$ is also $\V_0$-valued and, as we saw above, this implies that $L|_{\W\cap V_2}=0$.

\medskip

Write $\hat t=\beta(\hat w)$ and let $w_0\in\W\cap V_2$ be such that $\beta(w_0)=1$; then for every $h\in\R$
\begin{equation}\label{eq_blowup}
\phi_{\hat p,\lambda_n}(hw_0)=(\text{sgn }h)|h|^{1/2}\Delta(\hat t+h/\lambda_n^2,\hat t\,)v_0.
\end{equation}
By~\eqref{eq_12Holder2} there exists a sequence $(h_n)_n$ such that for every $n$
\[
\left|{h_n}\right|\in[\,c_1,1\,]\qquad\text{and}\qquad \|\phi_{\hat p,\lambda_n}(h_nw_0)\|\geq \sqrt{c_1}\:c_2\|v_0\|/2.
\]
Up to passing to a subsequence we can also assume that ${h_n}\to\bar h$ with $|\bar h|\in[\,c_1,1\,]$; since  
\begin{align*}
\|\phi_{\hat p,\lambda_n}(h_n w_0)-\phi_{\hat p,\lambda_n}(\bar h w_0)\|&=\left| \frac{u(\hat t+h_n/\lambda_n^2)-u(\hat t+\bar h/\lambda_n^2)}{1/\lambda_n}\right|\|v_0\|\\
&\leq|h_n-\bar h|^{1/2}\|v_0\|
\end{align*}
we obtain
\[
\|L(\bar h w_0)\|=\lim_n\|\phi_{\hat p,\lambda_n}(\bar h w_0)\|=\lim_n\|\phi_{\hat p,\lambda_n}(h_n w_0)\|\geq \sqrt{c_1}\:c_2\|v_0\|/2.
\]
This contradicts the fact that $L(\bar h w_0)=0$, and the claim is proved.

\medskip
	
	{\bf Claim 3:} for  $p\in\Gamma_\phi$, there exist infinitely many different blow-ups of $\Gamma_\phi$ at $p$.\\
Let $\hat p=\hat w\phi(\hat w)\in\Gamma_\phi$ be fixed and let $\hat t=\beta(\hat w)$; as before, fix also $w_0\in\W\cap V_2$ such that $\beta(w_0)=1$.		
By~\eqref{eq_12Holder2} we can find infinitesimal sequences $(s_n^1)_n,(s_n^2)_n$ such that
\[
\begin{array}{l}
\text{sgn}(s^1_n)=\text{sgn}(s^2_n)\text{ for every $n$},\vspace{.1cm}\\
\Delta(\hat t+s^1_n,\hat t\,)\geq \Delta(\hat t+s^2_n,\hat t\,) + c_2.
\end{array}
\]
Up to passing to a subsequence, we can assume that there exists $\sigma\in\{1,-1\}$ and $\Delta^1,\Delta^2\in\R$ such that
\[
\begin{array}{l}
\text{$\text{sgn}(s^1_n)=\text{sgn}(s^2_n)=\sigma$ for every $n$},\vspace{.1cm}\\
\Delta(\hat t+s^1_n,\hat t\,)\to\Delta^1\text{ and }\Delta(\hat t+s^2_n,\hat t\,)\to\Delta^2\text{ as }n\to\infty\vspace{.1cm},\\
\Delta^1\geq \Delta^2+c_2.
\end{array}  
\]
Due to the continuity of $s\mapsto \Delta(\hat t+s,\hat t\,)$ for $s\neq 0$, given $\Delta\in(\Delta^2,\Delta^1)$ one can find an infinitesimal sequence $(s_n)_n$ such that, for every $n$, $\text{sgn}(s_n)=\sigma$  and $\Delta(\hat t+s_n,\hat t\,)=\Delta$. Now, as in~\eqref{eq_blowup} the set $\delta_{|s_n|^{-1/2}}(\hat p^{-1}\Gamma_\phi)$ is the intrinsic graph of a map $\phi_{\hat p,|s_n|^{-1/2}}:\W\to\V$ such that
\[
\phi_{\hat p,|s_n|^{-1/2}}(\sigma w_0)=\sigma\Delta(\hat t+s_n,\hat t)v_0 =\sigma\Delta v_0.
\]
Since the family $(\phi_{\hat p,|s_n|^{-1/2}})_n$ is uniformly H\"older continuous, up to extracting a subsequence it converges locally uniformly to a map $\psi:\W\to\V$ such that $\psi(\sigma w_0) = \sigma \Delta v_0$. The arbitrariness of $\Delta\in(\Delta^2,\Delta^1)$ implies that there are infinitely many different blow-ups at $\hat p$, and this concludes the proof.
\end{proof}

\section*{Appendix}
We are now going to construct the function $u$ used in the proof of Theorem~\ref{thm:general}: this function, in a sense, provides a counter-example to a Rademacher property for Lipschitz functions from $(\R,\vert \cdot \vert^{1/2})$ to $(\R,\vert\cdot\vert)$. We will use a  classical procedure producing a self-similar function: although these ideas are well-known (see e.g.~\cite{Okamoto} and the references therein), we prefer to include a detailed construction because we were not able to find in the literature explicit statements for the precise estimates~\eqref{eq_12Holder_appendix2} we need.


We construct a function $u:[\,0,1\,]\to[\,0,1\,]$ whose difference quotients 
\[
\Delta(s,t)=\frac{u(s)-u(t)}{\text{sgn}(s-t)\vert s-t\vert^{1/2}}
\]
satisfy	
\begin{equation}\label{eq_12Holder_appendix1}
|\Delta(s,t)|\leq 1\qquad\text{for every $s,t\in[\,0,1\,]$.}
\end{equation}
We will construct $u$ in such a way that there exist $c_1>0$ and $c_2>0$ with the property that, for every $t\in[\,0,1\,]$ and $\delta\in(0,1\,]$, one can find $s_1,s_2\in[\,0,1\,]$ such that
	\begin{equation}\label{eq_12Holder_appendix2}
	\begin{array}{l}
	\text{sgn}(s_1-t)=\text{sgn}(s_2-t),\\
	c_1\delta\leq|s_1-t|\leq \delta,\\
	c_1\delta\leq|s_2-t|\leq \delta,\\
	|\Delta(s_1,t)-\Delta(s_2,t)|\geq c_2.
	\end{array}
	\end{equation}
We can extend $u$ to $\R$ by setting $u(t)=u(-t)$ for $t\in[\,-1,0\,]$ and $u(t+2n)=u(t)$ for all $n\in\bb Z$. The properties~\eqref{eq_12Holder_appendix1} and~\eqref{eq_12Holder_appendix2} imply the validity of~\eqref{eq_12Holder1} and~\eqref{eq_12Holder2} for the extended $u$.

 The function $u$ is obtained as the limit of a sequence $(u_n)_{n\in\mathbb N}$ where $u_0(t)=t$. The function $u_{n+1}$ is obtained from $u_n$ on setting
\begin{equation}\label{eq5fa5a8ca}
  u_{n+1}(t) = \begin{cases}
    \frac{2}{3}u_n\big(\frac{9}{4}t\big) &\text{ if } t\in \big [\,0,\frac{4}{9}\,\big],\\
    \frac{2}{3} -\frac{1}{3} u_n\big(9\big(t-\frac{4}{9}\big)\big) &\text{ if } t\in \big[\,\frac{4}{9},\frac{5}{9}\,\big],\\
    \frac{1}{3} + \frac{2}{3}u_n\big(\frac{9}{4}\big(t-\frac{5}{9}\big)\big) &\text{ if } t\in \big[\,\frac{5}{9},1\,\big].    
    \end{cases}
\end{equation}
The first few of the functions $u_0,u_1,u_2,\dots$ are plotted in Figure~\ref{fig5fa5a8a1}. Let us  notice that $u_n(0)=0$ and $u_n(1)=1$ for every $n$, hence $u_n(4/9)=2/3$ and $u_n(5/9)=1/3$ for every $n\geq 1$.

\begin{figure}
\includegraphics[width=\textwidth]{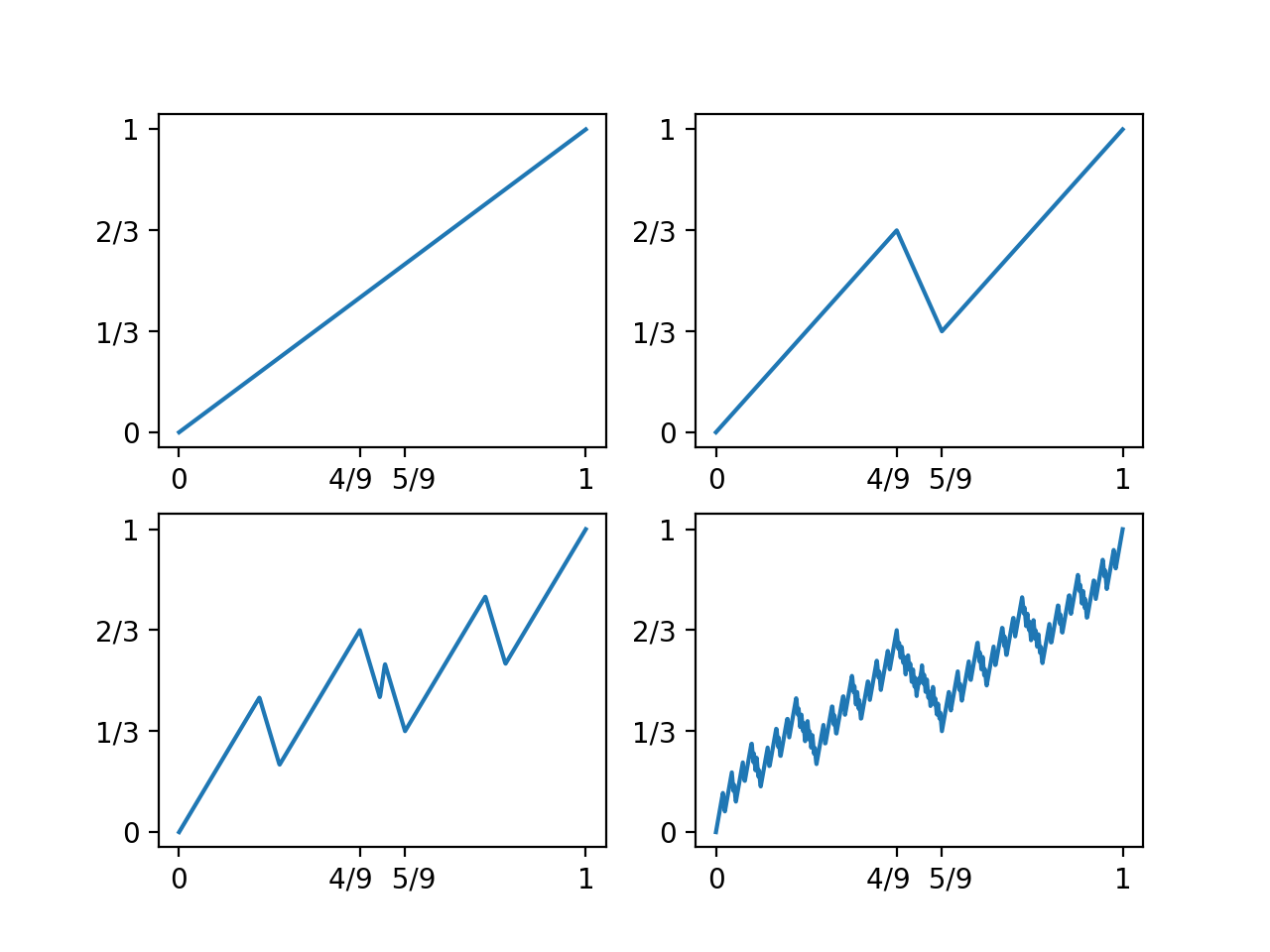}
\caption{Four instances of the functions $u_n$ defined in~\eqref{eq5fa5a8ca}}
\label{fig5fa5a8a1}
\end{figure}

Notice (see Figure~\ref{fig5fa5a890}) that the graph of $u_{n+1}$ is the union of three affine copies of the graph of $u_n$, via the following maps (acting on $p\in \R^2$):
\begin{equation}\label{eq5fa3b9c2}
	\begin{aligned}
	A_0(p)
		&= \begin{pmatrix}4/9 & 0 \\ 0 & 2/3\end{pmatrix} p,\\
	A_{4/9}(p) 
        &= \begin{pmatrix}1/9&0\\0&-1/3\end{pmatrix}p
        + \begin{pmatrix}4/9\\2/3\end{pmatrix} , \\
	A_{5/9}(p)
        &= \begin{pmatrix}4/9 & 0 \\ 0 & 2/3\end{pmatrix}p 
        + \begin{pmatrix}5/9\\1/3\end{pmatrix} .
	\end{aligned}
\end{equation}	
	\begin{figure}
\includegraphics[width=\textwidth]{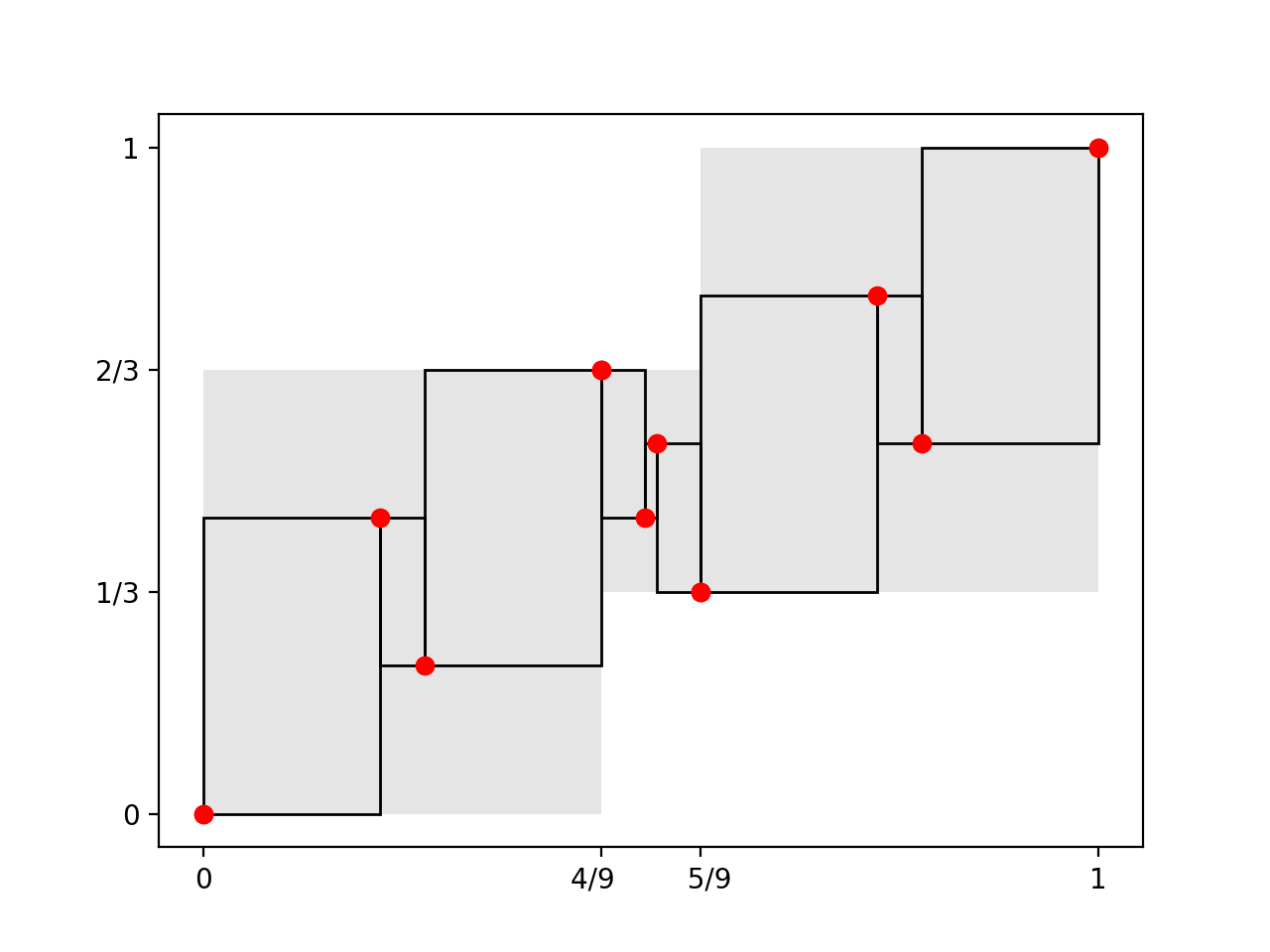}
\caption{Iterated images of the unit square under the affine maps in~\eqref{eq5fa3b9c2}; red dots are the images of $(0,0)$ and $(1,1)$, and they belong to the graph of the limit function $u$.}
\label{fig5fa5a890}
\end{figure}

{\bf Claim 1:}
   The functions $u_n$ 
   converge uniformly on $[\,0,1\,]$
   to a function $u$  for which~\eqref{eq_12Holder_appendix1} holds.

   
   The fact that $u_n$ uniformly converge to a continuous function $u$ is a consequence of the estimate 
   \begin{equation*}
   \|u_{n+1}-u_{n}\|_{C^0([\,0,1\,])} \le \frac23 \|u_{n}-u_{n-1}\|_{C^0([\,0,1\,])} .
   \end{equation*}
   This estimate follows directly from the definition~\eqref{eq5fa5a8ca}: for instance, for $t\in[\,0,4/9\,]$ one has
   \[
   \vert u_{n+1}(t)-u_{n}(t) \vert
   = \frac{2}{3} \vert u_n\big(9t/4\big) - u_{n-1}\big(9t/4\big) \vert
   \le \frac{2}{3} \|u_{n}-u_{n-1}\|_{C^0([\,0,1\,])} .
   \]
   Similarly, one can treat the other two cases $t\in[\,4/9,5/9\,]$ and $t\in[\,5/9,1\,]$.
   

   \medskip
   
 The bound~\eqref{eq_12Holder_appendix1} on the the difference quotients of $u$ follows from the fact that the same is true for all $u_n$ in the sequence, as we are now going to prove by induction on $n$.   
   The statement is clearly true for $n=0$. 
   Suppose that $u_n$ satisfies
   \[
   \text{ $ \vert u_{n}(t)-u_{n}(s)\vert \leq \vert t-s\vert^{1/2}$ for every $s,t\in[\,0,1\,]$;}
   \]
    we will prove that also $\vert u_{n+1}(t)-u_{n+1}(s)\vert \leq \vert t-s\vert^{1/2}$ for every $s,t\in[\,0,1\,]$. We distinguish several cases depending on which intervals ($[\,0,4/9\,]$, $[\,4/9,5/9\,]$ or $[\,5/9,1\,]$) the points $s$ and $t$ belong to. We can suppose that $s<t$.

{\bf Case 1:} $s$ and $t$ are in the same interval. We can use~\eqref{eq5fa5a8ca} and the induction hypothesis to conclude.

{\bf Case 2:} $s\in [\,0,4/9\,]$ and $t\in[\,4/9,5/9\,]$. 
Since $0\le u_n\le 1$, one sees from the definition of $u_{n+1}$ that 
$\max(u_{n+1}(s),u_{n+1}(t))\leq 2/3 = u_{n+1}(4/9) $. Thus
\begin{align*}
  \vert u_{n+1}(t)-u_{n+1}(s) \vert &\leq \max(u_{n+1}(4/9)-u_{n+1}(t), u_{n+1}(4/9)-u_{n+1}(s))\\
  &\leq \max((t-4/9)^{1/2}, (4/9-s)^{1/2}) \leq (t-s)^{1/2},
\end{align*}
 where the second inequality follows frow Case 1.

{\bf Case 3:} $s\in [\,4/9,5/9\,]$ and $t\in[\,5/9,1\,]$. This is similar to Case 2.

{\bf Case 4:} $s\in [\,0,4/9\,]$ and $t\in[\,5/9,1\,]$. Then either $\vert u_{n+1}(t)-u_{n+1}(s) \vert \leq 1/3$, and we are done because $\vert t-s\vert \geq 1/9$, or $\vert u_{n+1}(t)-u_{n+1}(s) \vert > 1/3$, and then necessarily  $u_{n+1}(s) < u_{n+1}(t)$ (otherwise, $0\leq u_{n+1}(s)-u_{n+1}(t)\leq u_{n+1}(4/9)-u_{n+1}(5/9)=2/3-1/3=1/3$)  and
\begin{align*}
  \vert u_{n+1}(t)-u_{n+1}(s) \vert &= u_{n+1}(t)-u_{n+1}(s) \\
                                    &= u_{n+1}(t)- u_{n+1}(5/9) -1/3 + u_{n+1}(4/9) -u_{n+1}(s)\\
                                    &\leq (t-5/9)^{1/2} -1/3 + (4/9-s)^{1/2}
\end{align*}
(by Case 1). By squaring the right hand side of the last inequality, we obtain

\begin{align*}
  \big ((t-5/9)^{1/2} -1/3 + (4/9-s)^{1/2}\big )^2 \hspace{-5cm}&\\
                                                                &= (t-s) + 2 (t-5/9)^{1/2} (4/9-s)^{1/2} -\frac{2}{3}(t-5/9)^{1/2} -\frac{2}{3}(4/9-s)^{1/2}\\
                                                                &= (t-s) +  (t-5/9)^{1/2} \big((4/9-s)^{1/2} -2/3 \big) + (4/9-s)^{1/2}\big ((t-5/9)^{1/2} -2/3\big) \\
  &\leq t-s,
\end{align*}
where we used the fact that  $4/9-s \leq 4/9$ and $t-5/9\leq 4/9$. This is enough to conclude.

{\bf Claim 2:} there exist $d_1>0$ and $d_2>0$ such that, for every $t_0\in[\,0,1\,]$, one can find $s_1,s_2\in[\,0,1\,]$ such that
\begin{equation}\label{eq_claim2d1d2}
	\begin{array}{l}
	\text{sgn}(s_1-t_0)=\text{sgn}(s_2-t_0)\\
	d_1\leq|s_1-t_0|\leq 1\\
	d_1\leq|s_2-t_0|\leq 1\\
	|\Delta(s_1,t_0)-\Delta(s_2,t_0)|\geq d_2.
	\end{array}
\end{equation}
In fact, we will prove Claim 2 for $d_1=1/18$ and
\[
d_2= \min\left\{ \frac13\left( \left(1-\tfrac 4{81}\right)^{-1/2} -1\right),\ \frac 79-\frac 35,\ \frac 1{\sqrt 5}\right\}.
\]
We distinguish several cases.

{\bf Case 1: $t_0\in[\,0,4/9\,]$.} In this case it suffices to consider $s_1=5/9$ and $s_2=1$, as we now show. Observe that the distances of $s_1,s_2$ from $t_0$ are both greater than $1/9>d_1$.\\
If $u(t_0)\geq 2/9$, by~\eqref{eq_12Holder_appendix1} and the equality $u(0)=0$ we have $t_0^{1/2}\geq u(t_0)$, hence $t_0\geq 4/81$; since $u(t_0)\leq 2/3$ we obtain
\[
\Delta(1,t_0)\geq \frac{\frac 13}{\sqrt{1-\tfrac 4{81}}}\qquad\text{and}\qquad
\Delta(5/9,t_0)\leq\frac{\tfrac13-\tfrac29}{\sqrt{\tfrac19}}=\frac13,
\]
so that $\Delta(1,t_0)-\Delta(5/9,t_0)\geq d_2$.\\
If $u(t_0)\leq 2/9$, then $(4/9-t_0)^{1/2}\geq 2/3- 2/9=4/9$, hence $5/9-t_0\geq1/9+(4/9)^2=(5/9)^2$ and
\[
\Delta(1,t_0)\geq  \frac 79 \qquad\text{and}\qquad \Delta(5/9,t_0)\leq\frac{\;\tfrac 13\;}{\frac59}=\frac 35
\]
and again $\Delta(1,t_0)-\Delta(5/9,t_0)\geq d_2$.

{\bf Case 2: $t_0\in[\,4/9,1/2\,]$.}
In this case we take $s_1=5/9$ and $s_2=1$. The distances of $s_1,s_2$ from $t_0$ are both no less than $1/18=d_1$ and, since $1/3\leq u(t_0)\leq2/3$, one gets
\[
\Delta(1,t_0)\geq \frac{\frac13}{\sqrt{\frac 59}}=\frac1{\sqrt 5} \qquad\text{and}\qquad \Delta(5/9,t_0)\leq 0.
\]

{\bf Case 3: $t_0\in[\,1/2,1\,]$.} We proved that, if $t_0\in[\,0,1/2\,]$, the claim can be proved on choosing $s_1=5/9$ and $s_2=1$. Therefore, due to the symmetry $u(x)=1-u(1-x)$, when $t_0\in[\,1/2,1\,]$ it is enough to take $s_1=0$ and $s_2=4/9$.

\medskip

{\bf Claim 3:} there exist $c_1>0$ and $c_2>0$ such that, for every $t\in[\,0,1\,]$ and $\delta\in(0,1\,]$, one can find $s_1,s_2\in[\,0,1\,]$ for which~\eqref{eq_12Holder_appendix2} holds.

By self-similarity, the graph of $u$  over the interval $ [\,0,1\,]\cap [\,t-\delta,t+\delta\,]$ contains the image of the graph of $u$ over $[\,0,1\,]$ under an affine map $L:\R^2\to\R^2$ which is a finite composition $L=A_{j_1}\circ\dots\circ A_{j_N}$ of maps $(A_{j_k})_{k=1,\dots,N}$ for $j_k$ in $\{0,4/9,5/9\}$. Observe that $L$ is an affine map of the form $L(x,y)=(L_1(x),L_2(y))$ for suitable affine maps $L_1,L_2:\R\to\R$ and it is not restrictive to assume that the length of the interval $L_1([\,0,1\,])$, which is contained in $[\,t-\delta,t+\delta\,]$, is at least $\delta/9$: this implies that there exists $\delta/9\leq c\leq \delta$ such that $|L_1(x)-L_1(y)|=c|x-y|$ for every $x,y\in [\,0,1\,]$. Let $t_0\in[\,0,1\,]$ be such that  $L(t_0, u(t_0))=(t,u(t))$. If $s_1,s_2\in[\,0,1\,]$ are such that~\eqref{eq_claim2d1d2} holds, then we have
\[
\begin{array}{l}
\text{sgn}(L_1(s_1)-t)=\text{sgn}(L_1(s_2)-t),\vspace{.2cm}\\
	\tfrac{d_1}{9}\delta\leq|L_1(s_1)-t|\leq \delta,\vspace{.2cm}\\
	\tfrac{d_1}{9}\delta\leq|L_1(s_2)-t|\leq \delta.
	\end{array}
\]
Since the maps $A_{j_k}$ do not modify the difference quotients, $L$ also has this property, i.e.,
\[
|\Delta(L_1(s_1),t)-\Delta(L_1(s_2),t)|=|\Delta(s_1,t_0)-\Delta(s_2,t_0)|\geq d_2.
\]
This concludes the proof.

\bibliographystyle{abbrv}
\bibliography{RefsCarnot-arxiv-2021-01-08}

\begin{thebibliography}{10}

\bibitem{MR2496564}
L.~Ambrosio, B.~Kleiner, and E.~Le~Donne.
\newblock Rectifiability of sets of finite perimeter in {C}arnot groups:
  existence of a tangent hyperplane.
\newblock {\em J. Geom. Anal.}, 19(3):509--540, 2009.

\bibitem{2020arXiv200602782A}
G.~{Antonelli} and A.~{Merlo}.
\newblock {Intrinsically Lipschitz functions with normal target in Carnot
  groups}.
\newblock {\em arXiv e-prints}, page arXiv:2006.02782, 2020.

\bibitem{antonelli2020rectifiable}
G.~{Antonelli} and A.~{Merlo}.
\newblock On rectifiable measures in {C}arnot groups: structure theory.
\newblock {\em arXiv e-prints}, page arXiv:2009.13941, 2020.

\bibitem{FrMaSe2014Differentiability}
B.~Franchi, M.~Marchi, and R.~P. Serapioni.
\newblock Differentiability and approximate differentiability for intrinsic
  {L}ipschitz functions in {C}arnot groups and a {R}ademacher theorem.
\newblock {\em Anal. Geom. Metr. Spaces}, 2:258--281, 2014.

\bibitem{FSSCJNCA}
B.~Franchi, R.~Serapioni, and F.~Serra~Cassano.
\newblock Intrinsic {L}ipschitz graphs in {H}eisenberg groups.
\newblock {\em J. Nonlinear Convex Anal.}, 7(3):423--441, 2006.

\bibitem{MR2836591}
B.~Franchi, R.~Serapioni, and F.~Serra~Cassano.
\newblock Differentiability of intrinsic {L}ipschitz functions within
  {H}eisenberg groups.
\newblock {\em J. Geom. Anal.}, 21(4):1044--1084, 2011.

\bibitem{FranchiSerapioni2016Intrinsic}
B.~Franchi and R.~P. Serapioni.
\newblock Intrinsic {L}ipschitz graphs within {C}arnot groups.
\newblock {\em J. Geom. Anal.}, 26(3):1946--1994, 2016.

\bibitem{2020arXiv200402520J}
A.~{Julia}, S.~{Nicolussi Golo}, and D.~{Vittone}.
\newblock {Area of intrinsic graphs and coarea formula in Carnot Groups}.
\newblock {\em arXiv e-prints}, page arXiv:2004.02520, 2020.

\bibitem{LDM_semigenerated}
E.~{Le Donne} and T.~Moisala.
\newblock Semigenerated step-3 {C}arnot algebras and applications to
  sub-{R}iemannian perimeter, 2020.

\bibitem{Okamoto}
H.~Okamoto.
\newblock A remark on continuous, nowhere differentiable functions.
\newblock {\em Proc. Japan Acad. Ser. A Math. Sci.}, 81(3):47--50, 2005.

\bibitem{VSNS}
D.~Vittone.
\newblock Lipschitz surfaces, perimeter and trace theorems for {BV} functions
  in {C}arnot-{C}arath\'{e}odory spaces.
\newblock {\em Ann. Sc. Norm. Super. Pisa Cl. Sci. (5)}, 11(4):939--998, 2012.

\bibitem{2020arXiv200714286V}
D.~{Vittone}.
\newblock {Lipschitz graphs and currents in Heisenberg groups}.
\newblock {\em arXiv e-prints}, page arXiv:2007.14286, 2020.

\end{thebibliography}

\end{document}